\documentclass{article}

\usepackage{amsmath}
\usepackage{amssymb}
\usepackage{amsthm}
\usepackage{bussproofs}

\usepackage{cleveref}

\theoremstyle{definition}
\newtheorem*{definition}{Definition}
\newtheorem*{theorem}{Theorem}
\newtheorem*{lemma}{Lemma}
\newtheorem*{corollary}{Corollary}

\newtheorem*{remark}{Remark}

\newcommand{\JConst}{\mathsf{JConst}}
\newcommand{\JVar}{\mathsf{JVar}}
\newcommand{\GTm}{\mathsf{GTm}}
\newcommand{\Tm}{\mathsf{Tm}}
\newcommand{\Prop}{\mathsf{Prop}}
\newcommand{\Fml}{\mathsf{Fml}}
\newcommand{\FmlS}{\mathsf{Fml_S}}
\newcommand{\axiom}[1]{\textbf{#1}}

\newcommand{\K}{\mathsf{K}}
\newcommand{\GK}{\mathsf{GK}}
\newcommand{\SE}{\mathsf{SE}}
\newcommand{\CL}{\mathsf{CL}}
\newcommand{\true}{\mathbb{T}}
\newcommand{\false}{\mathbb{F}}

\newcommand{\FmlSTm}{\Fml_{S_\Tm}}

\author{Michael Baur \and Thomas Studer}
\title{Semirings of Evidence}
\date{Institute of Computer Science\\
University of Bern,
Switzerland}
\begin{document}
\maketitle

\begin{abstract}
In traditional justification logic, evidence terms have the syntactic form of polynomials, but they are not equipped with the corresponding algebraic structure.
We present a novel semantic approach to justification logic that models evidence by a semiring.  Hence justification terms can be interpreted as polynomial functions on that semiring. This provides an adequate semantics for evidence terms and clarifies the role of variables in justification logic. Moreover, the algebraic structure makes it possible to compute with evidence. Depending on the chosen semiring this can be used to model trust, probabilities, cost, etc. Last but not least the semiring approach seems promising for obtaining a realization procedure for modal fixed point logics.
\end{abstract}

\section{Introduction}

Justification logic replaces the $\Box$-operator from modal logic with explicit evidence terms~\cite{Art01BSL,artemovFittingBook,justificationLogic2019}. That is, instead of formulas $\Box A$, justification logic features formulas $t:A$, where $t$ encodes evidence for $A$. Depending on the context, the term $t$ may represent a formal proof of $A$~\cite{Art01BSL,weakArithm} or stand for an informal 
justification (like direct observation, public announcement, private communication, and so on) for an agent's knowledge or belief of~$A$.
With the introduction of possible world models, justification logic has become an important tool to discuss and analyze epistemic situations~\cite{Art06TCS,Art08RSL,ArtKuz14APALnonote,Baltag201449,Stu13JSL}.

The terms of justification logic represent explicit evidence for an agent's belief or knowledge.  Within justification logic, we can reason  about this evidence. For instance, we can track different pieces of evidence pertaining to the same fact, which is essential for distinguishing between factive and non-factive justifications. This is applied nicely in Artemov's analysis of Russel's Prime Minister example~\cite{Art10LNCS}. Evidence terms can also represent the reasoning process of an agent. Therefore, agents represented by justification logic systems are not logically omniscient according to certain complexity based logical omniscience tests~\cite{ArtKuz06CSL,ArtKuz09TARK,ArtKuz14APALnonote}.

In traditional justification logic, terms are built using the binary operations~$+$ (called sum) and $\cdot$ (called application) and maybe other additional operations. Thus terms have the syntactic form of polynomials and are, in the context of the Logic of Proofs, indeed called proof polynomials.

This syntactic structure of polynomials is essentially used in the proof of realization, which provides a procedure that, given a theorem of a modal logic, constructs a theorem of the corresponding justification logic by replacing each occurrence of $\Box$ with an adequate justification term~\cite{Art01BSL}.

The main contribution of the present paper is to look at the syntactic structure of justification terms \emph{algebraically}, that is, we interpret justifications by a semiring structure. The motivation for this is threefold:
\begin{enumerate}
\item It provides an appropriate semantics for variables in evidence terms.
It was always the idea in justification logic that terms with variables justify derivations from assumptions.  The variables represent the input values, i.e.,  (arbitrary) proofs of the assumptions~\cite{Art01BSL}. But this was not properly reflected in the semantics where usually variables are treated like constants: to each term (no matter whether it contains variables or not) some set of formulas is assigned.  In our semiring semantics, ground terms (i.e.~terms not containing variables) are interpreted as elements of a semiring and terms with variables are interpreted as polynomial functions on the given semiring of justifications. Thus terms with variables are adequately represented and the role of variables is clarified.

\item The algebraic structure of terms makes it possible to compute with justifications.
Depending on the choice of the semiring, we can use the term structure to model levels of trust (Viterbi semiring), costs of obtaining knowledge (tropical semiring), probabilistic evidence (powerset semiring), fuzzy justifications (\L{}ukasiewicz semiring), and so on.
\item Considering $\omega$-continuous semirings, i.e.~semirings in which certain fixed points exist, may provide a solution to the problem of realizing modal fixed point logics like the logic of common knowledge. In these logics, some modal operators can be interpreted as fixed points of monotone operations. It seems likely that their realizations also should be fixed points of certain operations on semirings.
\end{enumerate}

\subsection*{Related Work}

Our approach is heavily inspired by the semiring approach for provenance in database systems~\cite{Green:2007}. There the idea is to label database tuples and to propagate expressions in order to annotate intermediate data and final outputs. One can then evaluate the provenance expressions in various semirings to obtain information about levels of trust, data prices,  required clearance levels, confidence scores, probability distributions, update propagation, and many more~\cite{Green:2017}.

This semiring framework has been adapted to many different query languages and data models. The core theoretical work of those approaches includes results on query containment, the construction of semirings, and fixed points~\cite{Amsterdamer:2011PIG,Amsterdamer:2011,Deutch:2015,Foster:2008,Geerts:2016,Green:2009:CCQ,Green2011}.

There are only few systems available where justification terms are equipped with additional structure. Two prominent examples are based on $\lambda$-terms (in contrast to the combinatory terms of the Logic of Proofs). The reflective lambda calculus~\cite{AltArt01PTCS} includes reduction rules on proof terms.  The intensional lambda calculus~\cite{ArtBon07LFCS} has axioms for evidence equality and also features a reduction relation on the terms. Another example is Krupski's recent work on sharp justification logics~\cite{Krupski2020}.

The present paper extends and corrects the results presented in~\cite{CLAR}. We extend our previous work by two realization theorems that establish the precise relationship between 
semirings of evidence and traditional epistemic modal logic, see Section~\ref{s:realization}. 
Moreover we slighlty changed the axiomatic system  in order to fix a mistake in the completeness proof.

\section{The Syntax of $\SE$}

We begin by defining the justification language as usual.
That is, we use a countable (infinite) set of constants $\JConst = \{0,1,c_1,c_2,...\}$ that includes two distinguished elements 0 and 1. Further we have a countable (infinite) set of variables $\JVar = \{x_1,x_2,...\}$.

\begin{definition}[Justification Term]
Justification terms are 
\[
c \in \JConst,\ x \in \JVar,\ s \cdot t,\ \text{ and }\ s+t, 
\]
where $s$ and $t$ are justification terms. The set of all justification terms is called\/ $\Tm$. A justification term that does not contain variables is called \emph{ground term}. $\GTm$ denotes the set of all ground terms. 
\end{definition}
Often we write only \emph{term} for \emph{justification term}.
Further, we need a countable set of atomic propostions
$\Prop = \{P_1,P_2,...\}$.

\begin{definition}[Formulas]
Formulas are $\perp$, $P$, $A \to B$ and $t:A$, where $t$ is a justification term, $P \in \Prop$ and $A$, $B$ are formulas. The set of all formulas is called $\Fml$.
\end{definition}
The remaining logical connectives $\lnot$, $\land$, $\lor$, and $\leftrightarrow$ are abbreviations as usual, e.g., $\lnot A$ stands for $A \to \bot$.

We will make use of substitutions to present the axioms of the logic~$\SE$.
Given a formula $A$ we write $A[w / t]$ for the result of simultaneously replacing all occurrences of the variable~$w$ in $A$ with the term~$t$. For instance, if $A$ is the formula $u:r\cdot w :B$, then 
$A[w/s+t]$ denotes the formula $u:r\cdot (s+t) :B$. For substituting all variables simultaneously we use a function $\sigma: \JVar \to \Tm$ by defining $A\sigma := A[x_1/y_1]...[x_n/y_n][y_1/\sigma(x_1)]...[y_n/\sigma(x_n)]$, where $x_1,...,x_n$ are the variables occurring in $A$ and the $y_i$ are fresh variables. We will use the same notations for substitutions in terms.

Now we can define a deductive system for the logic~$\SE$ about the semirings of evidence. It consists of the following axioms, where $w,x,y,z$ are variables and $A,B$ formulas.
\pagebreak

\noindent
The axioms of  $\SE$ are:\nopagebreak
\par\medskip
\nopagebreak
\begin{tabular}{ll}
\axiom{CL}\phantom{mm} & Every instance of a
propositional tautology \\
\axiom{j} & $x:(A \to B) \to (y:A \to x \cdot y:B)$\\
\axiom{j+} & $x:A \wedge y:A \to (x+y):A$\\
\axiom{a+} & $A[w/(x+y)+z] \to A[w/x+(y+z)]$\\
\axiom{c+} & $A[w/x+y] \to A[w/y+x]$\\
\axiom{0+} & $A[w/x+0] \leftrightarrow  A[w/x]$\\
\axiom{am} & $A[w/(x \cdot y) \cdot z] \leftrightarrow A[w/x \cdot (y \cdot z)]$\\
\axiom{a0} & $A[w/ x \cdot 0] \leftrightarrow A[w/0]$ \quad and \quad $A[w/ 0 \cdot x] \leftrightarrow A[w/0]$\\
\axiom{a1} & $ A[w/x \cdot 1] \leftrightarrow A[w/x]$ \quad  and \quad $ A[w/1 \cdot x] \leftrightarrow A[w/x]$\\
\axiom{dl} & $A[w/x \cdot (y+z)] \leftrightarrow A[w/x \cdot y + x \cdot z]$\\
\axiom{dr} & $A[w/(y+z) \cdot x] \leftrightarrow A[w/y \cdot x + z \cdot x]$
\end{tabular}
\par\medskip
\noindent
The rules of $\SE$ are:
\[
\text{\axiom{MP}} \quad \dfrac{A \hspace*{6mm} A \to B}{B}
\]
and 
\[
\text{\axiom{jv}} \quad \dfrac{A}{A[x/t]}
\]

The axiom schemes \axiom{a+}, \axiom{c+}, \axiom{0+}, \axiom{am}, \axiom{a0}, \axiom{a1}, \axiom{dl} and \axiom{dr} are called semiring axioms.
In the axiom scheme \axiom{j+}, we find
an important difference to traditional justification logic  where $\vee$ is used instead of $\wedge$, see also Section~\ref{s:realization} later.
The idea for $\axiom{j+}$ is to read $s+t:A$ as \emph{both $s$ and $t$ justify $A$.} This is useful, e.g., in the context of uncertain justifications where having two justifications is better than just having one.
The rule~\axiom{jv} shows the role of variables in $\SE$, which differs from traditional justification logic. In our approach a formula $A(x)$ being valid means that $A(x)$ is valid \emph{for all justifications~$x$}.

Now we show by an example how the semiring axioms work. Assume a formula $A$ contains an occurrence of $s+t$ (it may occur anywhere, even as a subterm of some other term). Starting from $A$, we want to derive the formula $B$, which is the same as $A$ except that the occurrence of $s+t$ is replaced by $t+s$.  So we let $C$ be the formula $A$ with this occurrence $s+t$ being replaced by a variable $x$ that doesn't occur in $A$. Now $C[x/s+t] \to C[x/t+s]$ is derived from an instance of the axiom scheme \textbf{c+} by \axiom{jv} and it is the same as $A \to B$.

Let us mention two immediate consequences of our axioms. First a version of axiom \axiom{0+} with  $x+0$ is replaced by $0+x$ is provable. Second, the direction from right to left in axiom \axiom{a+} is also provable.
\begin{lemma}\label{upl1}
The following formulas are derivable in $\SE$:
\[
A[w/0+x] \leftrightarrow A[w/x] 
\qquad\text{and}\qquad 
A[w/x+(y+z)] \leftrightarrow A[w/(x+y)+z].
\]
\end{lemma}

A theory is just any set of formulas.

\begin{definition}[Theory]
A theory $T$ is a subset of\/ $\Fml$. We use $T \vdash_{\SE} F$ to express that $F$ is derivable from $T$ in $\SE$. 
\end{definition}
Often we drop the subscript $_{\SE}$ in $\vdash_{\SE}$ when it is clear from the context. Moreover, we use $\vdash_{\CL}$ for the derivability relation in classical propositional logic.

A theory can compensate for the absence of constant specifications.
Usually, systems of justification logic are parametrized by a constant specification, i.e.,~a set containing pairs of constants and axioms. One then has a rule saying that a formula $c:A$ is derivable if $(c,A)$ is an element of the constant specification.
Here we do not adopt this approach but simply use a theory that includes $c:A$.
 
 \begin{definition}
	A theory $T$ is called axiomatically appropriate if
	\begin{enumerate}
		\item for all axioms $A$ there exists $c \in \JConst$ with $c:A \in T$
		\item for all $B \in T$ there exists $c \in \JConst$ with $c:B \in T$.
	\end{enumerate}
\end{definition}

Intuitively, in an axiomatically appropriate theory, all axioms have a justification and also all elements of the theory have a justification.
Using axiomatically appropriate theories, we get an analogue of modal necessitation in $\SE$.

\begin{lemma}[Internalization]
Let $T$ be an axiomatically appropriate theory. For any formula $A$, there exists a ground term $t$ such that
\[
T \vdash A  \quad\text{implies}\quad T \vdash t:A.
\]
\end{lemma}
\begin{proof}
	Induction on a derivation of $A$.
	\begin{enumerate}
		\item $A \in T$ or $A$ is an axiom: $t$ exists by the definition of an axiomatically appropriate theory.
		\item If $A$ is obtained by \textbf{MP} from $B \to A$ and $B$, then by I.H.~exist $t_1,t_2 \in \GTm$ such that $T \vdash_{\SE} t_1:(B \to A)$ and $T \vdash_{\SE} t_2:B$. Thus $T \vdash_{\SE} t:A$ holds for $t=t_1 \cdot t_2$.
		\item If $A[x/s]$ is obtained by \textbf{jv} from $A$, then by I.H.~exists $t \in \GTm$ such that $T \vdash_{\SE} t:A$. Finally, \textbf{jv} implies $T \vdash_{\SE} t:A[x/s]$.
		\qedhere
	\end{enumerate}
\end{proof}

\begin{remark}[Substitution]
In order to replace variables with terms, we do not need any properties of the theory $T$. In particular, we do not require it to be schematic, see Definition~\ref{def:schematic:1}. The implication \[T \vdash_{\SE} A \quad\text{implies}\quad T \vdash_{\SE} A[x/t]\] follows directly from rule \textbf{jv}.
\end{remark}
 
 Next we show that $\SE$ is a conservative extension of classical propositional logic, which implies consistency of  $\SE$.
 
\begin{theorem}[Conservativity]
The logic\/ $\SE$ is a conservative extension of classical propositional logic, $\CL$, i.e., for all formulas $A$ of the language of\/ $\CL$, we have
\[
\vdash_{\SE} A \quad\text{if{f}}\quad \vdash_{\CL} A.
\]
\end{theorem}
\begin{proof}
The claim from right to left is trivial as $\SE$ extends $\CL$. For the direction from left to right, we consider a mapping $^\circ$ from $\Fml$ to formulas of $\CL$ that simply drops all occurrences of $t:$. In particular, for any $\CL$-formula $A$, we have $A^\circ = A$.
Now it is easy to prove by induction on the length of $\SE$ derivations that for all $A \in \Fml$,
\[
\vdash_{\SE} A \quad\text{implies}\quad \vdash_{\CL} A^\circ.
\]
Simply observe that for any axiom $A$ of $\SE$, $A^\circ$ is a propositional tautology, and that the rules of $\SE$ respect the  $^\circ$-translation.
\end{proof}

Now consistency of $\SE$ follows immediately.

\begin{corollary}[Consistency of $\SE$]
The logic $\SE$ is consistent.
\end{corollary}
\begin{proof}
Assume towards contradiction that $\vdash_{\SE} \perp$. By conservativity of $\SE$ over $\CL$ we get $\vdash_{\CL} \perp$, which is a contradiction. 
\end{proof}

\begin{remark}\label{rem:dedTH:1}
The deduction theorem does not hold in $\SE$. This is due to possible occurrences of variables. For example $\{x:P\} \vdash 0:P$ says that if every term justifies $P$ then also $0$ justifies $P$, which is trivial. However $\vdash x:P \to 0:P$ is not valid because it can be shown that $\nvdash 1:P \to 0:P$.
\end{remark}

\section{The Semantics of $\SE$}

Our semantics of $\SE$ is similar to traditional semantics for justification logic in the sense that $t:A$ is given meaning by making use of an evidence relation.
Usually, this evidence relation assigns to each term a set of formulas, i.e.~the formulas that are justified by the term.
The novelty of our approach is that the evidence relation maps elements of a semiring to sets of formulas and terms are interpreted by the elements of this semiring.

\begin{definition}[Semiring]\label{d:semiring:1}
$K = (S, +, \cdot, 0, 1)$, where $S$ is the domain, is a semiring, if for all $a,b,c \in S$:
\begin{enumerate}
\item $(a+b)+c = a+(b+c)$, \quad $a+b=b+a$  \quad and \quad $a+0=0+a=a$ 
\item $(a\cdot b)\cdot c = a\cdot (b\cdot c)$  \quad and  \quad $a\cdot 1 = 1\cdot a = a$ 
\item $(a+b)\cdot c = a\cdot c + b\cdot c$  \quad and  \quad $c\cdot (a+b) = c\cdot a + c\cdot b$ 
\item $a \cdot 0 = 0 \cdot a = 0$
\end{enumerate}
\end{definition}

Thus, unlike in a ring, there is no inverse to $+$. We also do not require $\cdot$ to be commutative. 

Note that we use $+$ and $\cdot$ both as symbols in our language of justification logic and as operations in the semiring. It will always be clear from the context which of the two is meant.

For the following, assume we are given a semiring $K = (S, +, \cdot, 0, 1)$.
We use a function $I: \JConst \to S$ to map the constants of the language of~$\SE$ to the domain~$S$ of the semiring. We call this function $I$ an \emph{interpretation} if $I(0)=0$ and $I(1)=1$. 
We now extend $I$ to a homomorphism such that  $I: \Tm \to S[\JVar]$, where $S[\JVar]$ is the polynomial semiring in $\JVar$ over $S$ by setting:
\begin{enumerate}
\item $I(x):=x$ for variables $x$
\item $I(s+t) := I(s)+I(t)$ for terms $s,t$
\item $I(s\cdot t) := I(s) \cdot I(t)$ for terms $s,t$
\end{enumerate}
Let $K = (S, +, \cdot, 0, 1)$ be a semiring with domain $S$.
We define $\FmlS$ as the set of formulas where 
we use elements of $S$ instead of justification terms.
\begin{enumerate}
	\item $\perp \in \FmlS$
	\item $P \in \FmlS$, where $P \in \Prop$
	\item $A \to B \in \FmlS$, where $A \in \FmlS$ and $B \in \FmlS$
	\item $s:A \in \FmlS$, where $A \in \FmlS$ and $s \in S$
\end{enumerate}

\begin{definition}[Evidence relation]\label{d:er:1}
Let $S$ be the domain of a semiring. We call
$J \subseteq S \times \FmlS$ an \emph{evidence relation} if for all $s,t \in S$ and all $A,B \in \FmlS$:
\begin{enumerate}
\item $J(s, A \to B)$ and $J(t,A)$ imply $J(s\cdot t,B)$
\item $J(s,A)$ and $J(t,A)$ imply $J(s+t,A)$
\end{enumerate}
\end{definition}

\begin{definition}[Valuation]
A valuation $v$ is a function from $\JVar$ to $S$.
\end{definition}

The polynomial $I(t)$ can be viewed as a function $t_I: S^n \to S$, where $n$ is the number of variables that occur in $t$.
Hence, given an interpretation \mbox{$I: \JConst \to S$} and a valuation $v$, every $t \in \Tm$ can be mapped to an element $t_I(v(x_1),...,v(x_n))$ in $S$, which we denote by $t_I^v$. 
By abuse of notation, we only mention the variables that occur in the term $t$.
For a variable $x$ we have 
\[
v(x) = x_I(v(x)) = x_I^v.
\]
Given the definition of the polynomial function $t_I$, we find, e.g., 
\begin{equation}\label{eq:dot:1}
x_I^v \cdot y_I^v = x_I(v(x)) \cdot y_ I(v(y)) =  (x \cdot y)_I (v(x) , v(y))  = (x \cdot y)_I^v.
\end{equation}
For $A \in \Fml$ we define $A_I^v \in \FmlS$ inductively:
\begin{enumerate}
	\item $\perp_I^v := \perp$
	\item $P_I^v := P$, where $P \in \Prop$
	\item $(A \to B)_I^v := A_I^v \to B_I^v$, where $A \in \Fml$ and $B \in \Fml$
	\item $(s:A)_I^v := s_I^v : A_I^v$, where $A \in \Fml$ and $s \in \Tm$
\end{enumerate}
Let $A \in \Fml$ and let $x_1,\ldots,x_n$ be the variables that occur in $A$. Then $A_I$ denotes the function $A_I : S^n \to \FmlS$ defined by $A_I(y_1,...,y_n):=A_I^v$ where $v$ is such that $v(x_i)=y_i$.

\begin{definition}[Semiring model]
A \emph{semiring model} is a tuple  $M = (K, *, I, J)$ where
\begin{enumerate}
\item $K= (S, +, \cdot, 0, 1)$ is a semiring
\item $*$ is a truth assignment for atomic propositions, i.e., $*: \Prop \to \{\false,\true\}$
\item $I$ is an interpretation, i.e., $I: \JConst \to S$
\item $J$ is an evidence relation.
\end{enumerate}
\end{definition}

First we define truth in a semiring model for a given valuation. Because variables represent arbitrary justifications, we require a formula to be true for all valuations in order to be true in a semiring model. This means a formula with variables is interpreted as universally quantified.

\begin{definition}[Truth in a semiring model]
Let $M= (K, *, I, J)$ be a semi\-ring model, $v$ a valuation and $A$ a formula.
$M, v \Vdash A$ is defined as follows:
\begin{itemize}
\item $M, v \nVdash \perp$
\item $M, v \Vdash P$ \quad if{f} $\quad P^*=\true$
\item $M, v \Vdash A \to B$ \quad if{f} \quad  $M,v \nVdash A$ or $M,v \Vdash B$
\item $M, v \Vdash s:A$ \quad if{f} \quad  $J(s_I^v,A_I^v)$
\end{itemize}
Further we set $M \Vdash A$  \quad if{f} \quad  $M, v \Vdash A$ for all valuations $v$.
\end{definition}

For a semiring model $M$ and a theory $T$,
$M \Vdash T$ means $M \Vdash A$ for all $A \in T$.

\begin{definition}[Semantic consequence]
A theory $T$ \emph{entails} a formula $F$, in symbols
$T \Vdash F$, if for each semiring model~$M$ we have that 
\[
M \Vdash T \text{ implies } M \Vdash F.
\]
\end{definition}

By unfolding the definitions, we immediately get the following lemma, which is useful to establish soundness of $\SE$.

\begin{lemma}\label{l:one:1}
Let $M=(K,*,I,J)$ be a semiring model and let $v$ and $w$ be valuations with $v(x_i)=(t_i)_I^w$ for variables $x_i$ and terms $t_i$. Then
\[
M,v \Vdash A \quad\text{if{f}}\quad M,w \Vdash A\sigma \text{, where } \sigma(x_i)=t_i.
\]
\end{lemma}
\begin{proof}
By induction on the structure of $A$.
\begin{itemize}
\item Case $A=\perp$. We have $M,v \nVdash \perp$ and $M,w \nVdash \perp$.
\item Case $A=P$. We have $M,v \Vdash P \Leftrightarrow P^*=\true \Leftrightarrow M,w \Vdash P$.
\item Case $A=B \to C$. We have $M,v \Vdash B \to C\\ \Leftrightarrow M,v \nVdash B$ or $M,v \Vdash C\\ 
\stackrel{\text{I.H.}}{\Leftrightarrow} M,w \nVdash B\sigma$ or $M,w \Vdash C\sigma \\
\Leftrightarrow M,w \Vdash B\sigma \to C\sigma \Leftrightarrow M,w \Vdash (B \to C)\sigma$.
\item Case $A=s:B$. We have $M,v \Vdash s:B
\\\Leftrightarrow J(s_I^v, B_I^v) 
\\\Leftrightarrow J(s_I(v(x_1),...,v(x_n)), B_I(v(x_1),...,v(x_n))) 
\\\Leftrightarrow J(s_I((t_1)_I^w,...,(t_n)_I^w), B_I((t_1)_I^w,...,(t_n)_I^w)) 
\\\Leftrightarrow J((s\sigma)_I^w, (B\sigma)_I^w) 
\Leftrightarrow M,w \Vdash s\sigma:B\sigma 
\\\Leftrightarrow M,w \Vdash (s:B)\sigma$. \qedhere
\end{itemize}
\end{proof}

\begin{theorem}[Soundness]
Let $T$ be an arbitrary theory. Then:
\[
T \vdash F \quad\text{implies}\quad T \Vdash F.
\]
\end{theorem}
\begin{proof}
As usual by induction on the length of the derivation of $F$. Let $M=(K,*,I,J)$ be a semiring model such that $M \Vdash T$.
To establish our claim when $F$ is an axiom or an element of $T$, we let $v$ be an arbitrary valuation and show $M,v \Vdash F$ for the following cases:
\begin{enumerate}
\item $F \in T$. Trivial.
\item \axiom{CL}. Trivial.
\item \axiom{j}. Assume $M,v \Vdash x:(A \to B)$ and $M,v \Vdash y:A$.
That is  $J(x_I^v, (A \to B)_I^v)$ and $J(y_I^v, A_I^v)$ hold, which by Definition~\ref{d:er:1} implies
$J(x_I^v \cdot y_I^v, B_I^v)$.
Hence by~\eqref{eq:dot:1} we get $J((x \cdot y)_I^v, B_I^v)$, which yields
$ M,v \Vdash x \cdot y:B$.
\item \axiom{j+}. Similar to the previous case.

\item For the semiring axioms we prove
\[
M,v \Vdash A[x/s] \Leftrightarrow M,v \Vdash A[x/t] \text{ for all formulas $A$,}\]
where $s_I^v=t_I^v$ by induction on the structure of $A$:
\begin{itemize}
\item $A=\perp$ or $A=P$. Trivial.
\item $A=B \to C$. $M,v \Vdash (B \to C)[x/s]$
$\\\Leftrightarrow M,v \Vdash B[x/s] \to C[x/s]$
$\\\Leftrightarrow M,v \nVdash B[x/s]$ or $M,v \Vdash C[x/s]$
$\\\stackrel{\text{I.H.}}{\Leftrightarrow} M,v \nVdash B[x/t]$ or $M,v \Vdash C[x/t]$
$\\\Leftrightarrow M,v \Vdash B[x/t] \to C[x/t]$
$\\\Leftrightarrow M,v \Vdash (B \to C)[x/t]$.
\item $A=u:B$. $M,v \Vdash (u:B)[x/s]$
$\\\Leftrightarrow M,v \Vdash u[x/s]:B[x/s]$
$\\\Leftrightarrow J(u[x/s]_I^v,B[x/s]_I^v)$
$\\\Leftrightarrow J(u[x/t]_I^v,B[x/t]_I^v)$
$\\\Leftrightarrow M,v \Vdash u[x/t]:B[x/t]$
$\\\Leftrightarrow M,v \Vdash (u:B)[x/t]$.
\end{itemize}
By mentioning that all the semiring axioms have the form $A[x/s] \to A[x/t]$ or $A[x/s] \leftrightarrow A[x/t]$ with $s_I^v=t_I^v$, we finish this case.
\end{enumerate}

The case when $F$ has been derived by \axiom{MP} follows by I.H.~as usual.
For the case when $F = A[x/t]$ has been derived from $A$ by \axiom{jv}, we find by I.H.~that $M \Vdash A$, which is 
\begin{equation}\label{eq:sound:1}
M,v \Vdash A \text{ for all valuations $v$.}
\end{equation}
Given the term $t$ and an arbitrary valuation $w$, we find that there exists a valuation $v$ such that $v(x)=t_I^w$ and $v(y)=w(y)$ for all $y \neq x$.
By Lemma~\ref{l:one:1} we get 
\[
M,v \Vdash A \quad \text{if{f}} \quad M,w \Vdash A[x/t].
\]
Thus using \eqref{eq:sound:1}, we obtain $M,w \Vdash A[x/t]$. Since $w$ was arbitrary, we conclude $M \Vdash A[x/t]$.
\end{proof}

Next we establish completeness of $\SE$ with respect to semiring models.
For the completeness proof, we consider the free semiring over $\JConst \cup \JVar$.
We have for $s,t \in \Tm$:
\begin{itemize}
\item $[t]$ is the equivalence class of $t$ with respect to the semiring equalities, see Definition~\ref{d:semiring:1};
\item $[s]+[t]:=[s+t]$;
\item $[s]\cdot [t] := [s\cdot t]$;
\item $S_{\Tm} := \{[t]: t \in \Tm\}$;
\item $K_{\Tm} := (S_{\Tm}, +, \cdot, [0], [1])$ is the free semiring over $\JConst \cup \JVar$.
\end{itemize}

The following lemma states that $\SE$ respects the semiring equalities.

\begin{lemma}\label{eot}
Let $T$ be a theory and $s, t$ be terms with $[s]=[t]$. For each formula~$A$ we have
\[
T \vdash_{\SE} A[w/s] \quad\text{if{f}}\quad T \vdash_{\SE} A[w/t].
\]
\end{lemma}
Assume we are given an interpretation $I$ that maps constants to their equivalence class and a valuation $v$ that maps each variable $x_i$ to the equivalence class $[t_i]$ of some term $t_i$. 
The following lemma states that the interpretation of a term $s$ under $I$ and $v$ is the equivalence class of $s$
with each $x_i$ being replaced by $t_i$.
\begin{lemma}\label{tlt}
Assume that we are given an interpretation $I: \JConst \to S_\Tm$ with $I(c) = [c]$, a term $s$, and a valuation $v: \JVar \to S_\Tm$ with $v(x_i) =[t_i]$. Then we have 
\[
s_I^v = [s\sigma] \text{, where } \sigma(x_i)=t_i.
\]
\end{lemma}

\begin{proof}
Induction on the structure of $s$:
\begin{itemize}
\item $c_I^v = [c]$ by definition of $I$.
\item $(x_i)_I^v = [t_i]$ by definition of $v$.
\item $(s_1+s_2)_I^v = (s_1)_I^v+(s_2)_I^v$.
	By I.H.~we have
	\[
	(s_1)_I^v =  [s_1\sigma] \quad\text{and}\quad (s_2)_I^v = [s_2\sigma].
	\]
	Thus $ (s_1)_I^v+(s_2)_I^v = [s_1\sigma]+[s_2\sigma] = [(s_1+s_2)\sigma]$.
\item $(s_1 \cdot s_2)_I^v = (s_1)_I^v \cdot (s_2)_I^v$.
	By I.H.~we have
	\[
	(s_1)_I^v =  [s_1\sigma] \quad\text{and}\quad (s_2)_I^v = [s_2\sigma].
	\]
	Thus $(s_1)_I^v \cdot (s_2)_I^v = [s_1\sigma] \cdot [s_2\sigma] = [(s_1 \cdot s_2)\sigma]$ \qedhere
\end{itemize}
\end{proof}

We extend the notion of equivalence to formulas by 
defining the function 
\[
[\cdot]: \Fml \to \FmlSTm
\]
as follows: 
\begin{itemize}
\item $[\perp] := \perp$
\item $[P] := P$
\item $[A \to B] := [A] \to [B]$
\item $[t:A] := [t]:[A]$
\end{itemize}

Intuitively $[A]$ is the formula where each justification term is replaced by its equivalence class in the free semiring.
Observe that if $I(c)=[c]$ and $v(x)=[x]$, then $[A]=A_I^v$. Now we extend \Cref{tlt} to formulas.

\begin{lemma}\label{tlf}
Assume that we are given the interpretation $I: \JConst \to S_\Tm$ with $I(c) = [c]$, a formula $A$, and a valuation $v: \JVar \to S_\Tm$ with $v(x_i) =[t_i]$. Then we have 
\[
A_I^v = [A\sigma] \text{, where } \sigma(x_i)=t_i.
\]
\end{lemma}

\begin{proof}
Induction on the structure of $A$:
\begin{itemize}
\item $\perp_I^v = \perp = [\perp]$
\item $P_I^v = P = [P]$
\item $(A \to B)_I^v = A_I^v \to B_I^v \\\stackrel{\text{I.H.}}{=} [A\sigma] \to [B\sigma] = [(A \to B)\sigma]$
\item $(s:A)_I^v = s_I^v:A_I^v \\\stackrel{\text{I.H.~and L.~\ref{tlt}}}{=} [s\sigma]:[A\sigma] = [(s:A)\sigma]$ \qedhere
\end{itemize}
\end{proof}

Let $\Prop2$ be an infinite set of atomic propositions with $\Prop \cap \Prop2$ being empty. Then there exists an bijective function $f: S_{\Tm} \times \FmlSTm \to \Prop2$. We assume $f$ to be fixed for the rest of this section. Based on this function we define a translation~$'$ that maps formulas of $\Fml$ to pure propositional formulas containing atomic propositions from $\Prop \cup \Prop2$.
\begin{enumerate}
	\item $\perp ' := \perp$
	\item $P' := P$
	\item $(A \to B)' := A' \to B'$
	\item $(t:A)' := f([t],[A])$
\end{enumerate}

Let $T$ be a theory. We define the corresponding theory:
\[T' := \{(A\sigma)' \mid A \in T \text{ or } A \text{ is an axiom of } \SE, \sigma: \JVar \to \Tm\}.\]
Suppose $A' \in T'$. Then there exist a formula $B$ with $B \in T$ or $B$ is an axiom and $\sigma: \JVar \to \Tm$ such that $(B\sigma)' = A'$. This implies $B\sigma[x/t]' = A[x/t]'$. Now we have $A[x/t]' \in T'$. Therefore the following implication is proved:
\begin{equation}\label{eq:T'}
A' \in T' \Rightarrow A[x/t]' \in T'
\end{equation}

In fact this does not only hold for formulas in $T'$ but also for all formulas derived from $T'$ by classical propositional logic.

\begin{lemma}
	If $T' \vdash_{\CL} A'$ then $T' \vdash_{\CL} A[x/t]'$.
\end{lemma}

\begin{proof}
	Induction on the derivation of $A'$.
	Note that $T'$ contains all the axioms of $\CL$. So we can omit this case.
	\begin{enumerate}
		\item If $A' \in T'$ then $A[x/t]' \in T'$ by the above observation. Thus we get $T' \vdash_{\CL} A[x/t]'$.
		\item If $A'$ is obtained by \textbf{MP} from $B$ and $B \to A'$ then $B$ can be written as $C'$ because $f$ is surjective. The induction hypothesis ($T' \vdash_{\CL} C[x/t]'$ and $T' \vdash_{\CL} C[x/t]' \to A[x/t]'$) yields $T' \vdash_{\CL} A[x/t]'$. \qedhere
	\end{enumerate}
\end{proof}

The translation~$'$ respects the derivability relation of $\SE$. Hence we have the following lemma.
\begin{lemma}\label{secll}
	$T \vdash_{\SE} A \Leftrightarrow T' \vdash_{\CL} A'$
\end{lemma}

\begin{proof}
	Left to right by induction on a derivation of $A$:
	\begin{enumerate}
		\item If $A \in T$ or $A$ is an axiom then $A' \in T'$ and therefore $T' \vdash_{\CL} A'$.
		\item If $A$ is obtained by \textbf{MP} from $B$ and $B \to A$ then the induction hypothesis ($T' \vdash_{\CL} B'$ and $T' \vdash_{\CL} B' \to A'$) immediately yields $T' \vdash_{\CL} A'$.
		\item If $A[x/t]$ is obtained by \textbf{jv} from $A$, then the I.H.~is $T' \vdash_{\CL} A'$. By the previous lemma we conclude $T' \vdash_{\CL} A[x/t]'$.
	\end{enumerate}
	Right to left by induction on a derivation of $A'$:
	\begin{enumerate}
		\item If $A' \in T'$ then there exist a formula $B$ with $B \in T$ or $B$ is an axiom and $\sigma: \JVar \to \Tm$ such that $(B\sigma)' = A'$. Trivially we have $T \vdash_{\SE} B$ and get by \textbf{jv} that $T \vdash_{\SE} B\sigma$. Since $f$ is injective, the only difference between $A$ and $B\sigma$ is that some terms may be replaced by equivalent ones (modulo the semiring). Therefore, we get $T \vdash_{\SE} A$ by using Lemma~\ref{eot}.
		\item If $A'$ is a propositional tautology then so is $A$ because $f$ is injective, but some terms in $A$ may be replaced by equivalent ones. We get $T \vdash_{\SE} A$ again by %
	 Lemma~\ref{eot}	
		and propositional reasoning.
		\item If $A'$ is obtained by \textbf{MP} from $B \to A'$ and $B$ then $B$ can be written as $C'$. The induction hypothesis ($T \vdash_{\SE} C \to A$ and $T \vdash_{\SE} C$) implies $T \vdash_{\SE} A$.
	\end{enumerate}\qedhere
\end{proof}

\Cref{secll} gives us the ability to switch from $\SE$ to $\CL$ and vice versa. Therefore, we can use completeness of $\CL$ to obtain completeness for $\SE$.

\begin{theorem}[Completeness]
Let $T$ be an arbitrary theory. Then:
\[T \Vdash F \quad\text{implies}\quad T \vdash F.\]
\end{theorem}

\begin{proof}
We will prove the contraposition, which means for $T \nvdash F$ we will construct a semiring model $M$ and find a valuation $v$, such that $M \Vdash T$ and $M,v \nVdash F$.
Assume $T \nvdash F$. By \Cref{secll} we get $T' \nvdash_\CL F'$. The completeness of $\CL$ delivers $*: \Prop \cup \Prop2 \to \{\mathbb{F},\mathbb{T}\}$, such that for the $\CL$-model $M_*$ consisting of $*$ we have $M_* \Vdash T'$ and $M_* \nVdash F'$. 
Now we can define the semiring model $M$:
\begin{itemize}
\item $M := (K_{\Tm}, *|_{\Prop}, I, J)$
\item $*|_{\Prop}$ is the restriction of $*$ to $\Prop$
\item $I: \JConst \to S_{\Tm}$, $I(c) := [c]$
\item $J := \{([t], [A]) \mid M_* \Vdash f([t], [A])\}$
\end{itemize}
In order to prove that $M$ is a semiring model, we need to show that $J$ is an evidence relation.
\begin{enumerate}
\item From $M_* \Vdash T'$ we derive $M_* \Vdash (s:(A \to B) \to (t:A \to s \cdot t:B))'$ $\forall s,t \in \Tm$ and $\forall A,B \in \Fml$ by using the definition of $T'$ and \eqref{eq:T'}. It follows
\[M_* \Vdash f([s],[A \to B]) \to (f([t],[A]) \to f([s \cdot t],[B])).\]
By the truth definition in $\CL$ we find
\[\text{if } f([s],[A \to B])^* = \mathbb{T} \text{ and } f([t],[A])^* = \mathbb{T} \text{ then } f([s \cdot t],[B])^* = \mathbb{T}.\]
From the definition of $J$ in $M$ we get
\[\text{if } J([s],[A] \to [B]) \text{ and } J([t],[A]) \text{ then } J([s] \cdot [t],[B]).\]
\item From $M_* \Vdash T'$ we derive $M_* \Vdash (s:A \wedge t:A \to s+t:A)'$ $\forall s,t \in \Tm$ and $\forall A \in \Fml$ by using the definition of $T'$ and \eqref{eq:T'}. It follows
\[M_* \Vdash f([s],[A]) \wedge f([t],[A]) \to f([s+t],[A]) \text{ } \forall s,t \in \Tm \text{ and } \forall A \in \Fml.\]
By the truth definition in $\CL$ we find
\[\text{if } f([s],[A])^* = \mathbb{T} \text{ and } f([t],[A])^* = \mathbb{T} \text{ then } f([s+t],[A])^* = \mathbb{T}.\]
From the definition of $J$ in $M$ we get
\[\text{if } J([s],[A]) \text{ and } J([t],[A]) \text{ then } J([s]+[t],[A]).\]
\end{enumerate}
Knowing that $M$ is a semiring model we prove 
\begin{equation}\label{eq:comp}
M_* \Vdash (A\sigma)' \Leftrightarrow M,w \Vdash A
\end{equation}
by induction on the structure of $A$, where $w(x_i)=[t_i]$ and $\sigma(x_i)=t_i$.
\begin{itemize}
\item Case $A=\perp$. We have $M_* \nVdash \perp'$ and $M,w \nVdash \perp$.
\item Case $A=P$. We have  $M_* \Vdash P' \Leftrightarrow M_* \Vdash P \Leftrightarrow P^*=\mathbb{T} \Leftrightarrow M,w \Vdash P$.
\item Case $A=B \to C$. We have $M_* \Vdash ((B \to C)\sigma)' 
\\\Leftrightarrow M_* \Vdash (B\sigma)' \to (C\sigma)' 
\\\Leftrightarrow M_* \nVdash (B\sigma)' \text{ or } M_* \Vdash (C\sigma)' 
\\\Leftrightarrow M,w \nVdash B \text{ or } M,w \Vdash C \hfill \text{(by induction hypothesis)}
\\\Leftrightarrow M,w \Vdash B \to C
\\\Leftrightarrow M,w \Vdash B \to C$.
\item Case $A=s:B$. We have  $M_* \Vdash ((s:B)\sigma)' 
\\\Leftrightarrow M_* \Vdash (s\sigma:B\sigma)' 
\\\Leftrightarrow M_* \Vdash f([s\sigma], [B\sigma]) \hfill \text{(by definition of }'\text{)}
\\\Leftrightarrow J([s\sigma], [B\sigma]) \hfill \text{(by definition of }J\text{)}
\\\Leftrightarrow J(s_I^w, [B\sigma]) \hfill \text{(by \Cref{tlt})}
\\\Leftrightarrow J(s_I^w, B_I^w) \hfill \text{(by \Cref{tlf})}
\\\Leftrightarrow M,w \Vdash s:B$
\end{itemize}
Now we show $M \Vdash T$, i.e. $M,w  \Vdash T$ for all valuations $w$.
Hence let $w$ be an arbitrary valuation (assume $w(x_i)=[t_i]$ and $\sigma(x_i)=t_i$) and $A \in T$. It follows $A' \in T'$ and by \eqref{eq:T'} also $(A\sigma)' \in T'$. From $M_* \Vdash T'$ we get $M_* \Vdash (A\sigma)'$. \eqref{eq:comp} implies $M,w \Vdash A$. Since $w$ was arbitrary we conclude $M \Vdash T$. 

Now we consider the special case of \eqref{eq:comp} where $w=v$ with  $v(x)=[x]$. We have
\[M_* \Vdash A' \Leftrightarrow M,v \Vdash A\]
Remembering $M_* \nVdash F'$, we derive $M,v \nVdash F$, which finishes the proof. 
\end{proof}

\section{Realization}\label{s:realization}

Realization is concerned with the relationship between justification logic and modal logic.
Replacing all terms in a formula of justification logic by $\Box$-operators yields a formulas of modal logic. This is called \emph{forgetful projection} since by this translation, one `forgets' the explicit evidence for one's beliefs.
It is fairly obvious that the forgetful projection of a theorem of justification logic yields a theorem of modal logic. This is so since the translation of any axiom of justification logic yields a theorem of modal logic and the translation of the rules also yields (derivable) rules of modal logic.

The converse direction, called \emph{realization}, is more interesting and also more difficult to establish. We show that, under certain natural conditions, there is a construction that replaces all modalities in a theorem of modal logic by justification terms such that the resulting formula is a theorem of justification logic.

In this section, we let $\circ$ be the mapping from $\Fml$ to formulas of the modal logic~$\K$ that replaces all occurrences of $s:$ in a formula of $\SE$ with $\Box$.

\begin{definition}
We say that a theory $T$ is \emph{pure} if for each formula $A \in T$ we have that $\vdash_{\K} A^\circ$.
\end{definition}

We immediately get the following lemma.
\begin{lemma}[Forgetful projection] 
Let $T$ be a pure theory.
For any formula $A$ we have
\[
T \vdash_{\SE} A \quad\text{implies}\quad \vdash_{\K} A^\circ.
\]
\end{lemma}

To investigate mappings from modal logic to $\SE$, we need the deductive system $\GK$ for the logic $\K$.

\begin{definition}
	If $\Gamma$ and $\Delta$ are multisets of formulas, then the expression $\Gamma \supset \Delta$ is a sequent.
\end{definition}

The axioms of $\GK$ are 
\[
P \supset P    \qquad\text{and}\qquad  \perp \supset
\]
 where $P \in \Prop$.
The rules of $\GK$ are the following:
\[
\axiom{$(\to \supset)$}\,
\dfrac{\Gamma \supset \Delta, A \hspace*{6mm} B,\Gamma \supset \Delta}{A \to B, \Gamma \supset \Delta}\,
\qquad\qquad
\axiom{$(\supset \to)$}   \dfrac{A,\Gamma \supset \Delta, B}{\Gamma \supset \Delta, A \to B} 
\]
\[
\axiom{$(\Box)$}\, \dfrac{\Gamma \supset A}{\Box \Gamma \supset \Box A}
\]
\[
\axiom{$(w \supset)$}\, \dfrac{\Gamma \supset \Delta}{A,\Gamma \supset \Delta} 
\qquad\qquad
\axiom{$(\supset w)$}\, \dfrac{\Gamma \supset \Delta}{\Gamma \supset \Delta,A} 
\]
\[
\axiom{$(c \supset)$}\, \dfrac{A,A,\Gamma \supset \Delta}{A,\Gamma \supset \Delta} 
\qquad\qquad
\axiom{$(\supset c)$}\, \dfrac{\Gamma \supset \Delta,A,A}{\Gamma \supset \Delta,A}
\]
System $\GK$ is sound and complete for the modal logic $\K$, see, e.g.,~\cite{mints92,ts2000}.

\begin{theorem}
	Let $A$ be a formula of the language of modal logic.
	Then $\vdash_{\GK} \supset A$ if and only if $A$ is valid in $\K$.
\end{theorem}

So instead of directly realising modal formulas, we realise sequents of $\GK$. Therefore we define what it means for a sequent to be derivable in $\SE$.

\begin{definition}\label{def:seqreal:1}
	Let $\Gamma \supset \Delta$ be a sequent where each $\Box$ is replaced by some term. $\Gamma \supset \Delta$ is called derivable in $\SE$ from a theory $T$ if $T \vdash_{\SE} \bigwedge \Gamma \to \bigvee \Delta$.
\end{definition}

In the traditional approach to constructive realization~\cite{Art01BSL,justificationLogic2019}, one would say that $\Gamma \supset \Delta$ is derivable in $\SE$ if $\bigwedge \Gamma \vdash_{\SE} \bigvee \Delta$. 
This does not work in the framework of $\SE$ since the deduction theorem does not hold for $\SE$ (see Remark~\ref{rem:dedTH:1}). Therefore, we use an approach according to Definition~\ref{def:seqreal:1}.

In the realization procedure, most of the effort goes into constructing terms for the Box-modalities introduced in the rule $(\Box)$, which is the next thing we are going to do. Because \textbf{j+} is formulated with $\wedge$ we can not use Artemov's original realization algorithm. But Kuznets~\cite{Kuz09PLS} found a realization procedure with the same ideas as Artemov except that justification terms are constructed without~$+$. It can be applied in the context of $\SE$.

For the following definition we need the notion of positive and negative occurrences of $\Box$ within a given modal formula $A$. First we assign a polarity to each subformula occurrence within $A$ as follows.
\begin{enumerate}
\item The only occurrence of $A$ within $A$ is given positive polarity.
\item If a polarity is already assigned to an occurrence $B \to C$ within $A$, then the same polarity is assigned to $C$ and the opposite polarity is assigned to $B$.
\item If a polarity is already assigned to an occurrence $\Box B$ within $A$, then the same polarity is assigned to $B$. 
\end{enumerate}
Now we assign a polarity to each occurence of $\Box$ as follows:
the leading $\Box$ in an occurrence of $\Box B$ within $A$ has the same polarity as the occurrence of $\Box B$ within $A$.

\begin{definition}
A \emph{realization} $r$ is a mapping from modal formulas to $\Fml$ such that for each modal formula $F$ we have $(r(F))^\circ =F$.
A realization is \emph{normal} if all negative occurrences of\/ $\Box$ are mapped to distinct justification variables.
\end{definition}

As usual, we need  a notion of schematicness to obtain a realization result.

\begin{definition}\label{def:schematic:1}
A theory $T$ is schematic if it satisfies the following property: for each constant $c \in \JConst$, the set of axioms $\{A \mid c:A \in T \text{ and } A \text{ is an axiom}\}$ consists of all instances of one or several (possibly zero) axiom schemes of $\SE$.
\end{definition}

\begin{lemma}\label{syllemma}
	Let $T$ be an axiomatically appropriate theory.\\
	If $T \vdash_{\SE} s:(A \to B)$ and $T \vdash_{\SE} t:(B \to C)$ then $T \vdash_{\SE} d \cdot c \cdot t \cdot s:(A \to C)$, where $T \vdash_{\SE} c:(B \to C \to (A \to (B \to C)))$ and\\ $T \vdash_{\SE} d:(A \to (B \to C) \to (A \to B \to (A \to C)))$.
\end{lemma}

\begin{proof}
	The proof is encoded as $(d \cdot (c \cdot t)) \cdot s$.  The constants $c$ and $d$ exist because~$T$ is axiomatically appropriate. 
\end{proof}

\begin{definition}
Let $T$ be an axiomatically appropriate theory and suppose 
\[
T \vdash_{\SE} s:(A \to B) \text{ and } T \vdash_{\SE} t:(B \to C).
\]
We define $syl(s,t)$ such that 
\begin{equation}\label{eq:syl:1}
T \vdash_{\SE} syl(s,t):(A \to C)
\end{equation}
and  $syl(s,t)$ is the least term (according to some given fixed ordering on terms) that satisfies~\eqref{eq:syl:1}.
\end{definition}

\begin{remark}
Lemma~\ref{syllemma} guarantees the existence of $syl(s,t)$.
Note that the term $syl(s,t)$ depends on the formulas $A,B$ and $C$. This dependency disappears if the theory $T$ is schematic.
\end{remark}
For the rest of this section we denote by $d_n$ a term such that for all formulas $A_1,...,A_n$ and $1 \leq i \leq n$
\begin{equation}\label{eq:dn:1}
T \vdash_{\SE} d_n:(A_i \to A_1 \vee ... \vee A_n).
\end{equation}

\begin{lemma}\label{orlemma}
	Let $T$ be an axiomatically appropriate theory and $n \in \mathbb{N}_{>0}$.\\
	Assume there exists $d_n$.\\
	If $T \vdash_{\SE} s_i:(A_i \to B)$ for all $i$, then for an arbitrary $t \in \Tm$
	\[T \vdash_{\SE} t:A_i \to syl(d_n, e_n \cdot s_1 \cdot ... \cdot s_n) \cdot t :B,\]
	where $T \vdash_{\SE} e_n:((A_1 \to B) \to (... \to ((A_n \to B) \to (A_1 \vee ... \vee A_n \to B))...))$.
\end{lemma}

\begin{proof}
The existence of $e_n$ follows from the internalization property.
By using \textbf{j} and \textbf{MP} $n$ times we get $T \vdash_{\SE} e_n \cdot s_1 \cdot ... \cdot s_n : (A_1 \vee ... \vee A_n \to B)$. Applying the $syl$-function constructed in \Cref{syllemma} gives $T \vdash_{\SE} syl(d_n, e_n\cdot s_1\cdot ...\cdot s_n):(A_i \to B)$ for all $i$. From \textbf{j} as \[syl(d_n, e_n\cdot s_1\cdot ...\cdot s_n):(A_i \to B) \to (t:A_i \to syl(d_n, e_n\cdot s_1\cdot ...\cdot s_n)\cdot t :B)\] we infer $T \vdash_{\SE} t:A_i \to syl(d_n, e_n \cdot s_1 \cdot ... \cdot s_n) \cdot t :B$ for all $i$. 
\end{proof}

In the definition of $d_n$ the brackets of the disjunction are missing, because they don't matter: If $d_n$ exists for one specific placement of the brackets then so does $e_n$ and we can apply the previous lemma.

The next lemma is an easy consequence of axiom \axiom{j}. Note that the existence of the constant $c_n$ follows from $A_1 \to (... \to (A_n \to A_1 \wedge ... \wedge A_n)...))$ being an instance of  \axiom{CL} and $T$ being axiomatically appropriate.
\begin{lemma}\label{andlemma}
	Let $T$ be an axiomatically appropriate theory and $n \in \mathbb{N}_{>0}$. Then \[T \vdash_{\SE} t_1:A_1 \wedge ... \wedge t_n:A_n \to c_n \cdot t_1 \cdot ... \cdot t_n:(A_1 \wedge ... \wedge A_n),\]
	where $T \vdash_{\SE} c_n:(A_1 \to (... \to (A_n \to A_1 \wedge ... \wedge A_n)...))$.
\end{lemma}

For a multiset $\Gamma$ and a variable $x$ define $x:\Gamma := \{x:A \mid A \in \Gamma\}$.
Further we define terms $x^k$ recursively by $x^0:=1$ and $x^{k+1}:=x^{k} \cdot x$. 
Remember that by axiom \axiom{am}, we have associativity of the application operation. Therefore, and by axiom \axiom{a1}, we may use, e.g., the term $c \cdot (( 1 \cdot x)  \cdot x)$ where one would expect  $(c \cdot x)  \cdot x$.

\begin{lemma}\label{rlemma}
	Let $T$ be an axiomatically appropriate theory, $n \in \mathbb{N}_{>0}$, and assume that $d_n$ exists as in~\eqref{eq:dn:1}. 
	Further assume that for all $i$ such that $1 \leq i \leq n$,  we are given 
	multisets $\Gamma_i \subseteq \Fml$ with $|\Gamma_i|=m$ such that
     $T \vdash_{\SE} \bigwedge \Gamma_i \to A$.\\% for all $1 \leq i \leq n$.\\
	Then there exists $q \in \Tm$ such that  for each $\Gamma_i$ we have $T \vdash_{\SE} \bigwedge x: \Gamma_i \to q:A$ where $x$ is a variable.
\end{lemma}

\begin{proof}
	By internalization we find  ground terms $s_i$ such that 
	\[
	T \vdash_{\SE} s_i:(\bigwedge \Gamma_i \to A).
	\]
	Because $T$ is axiomatically appropriate and we assume that $d_n$ exists, we can use \Cref{orlemma} and get for an arbitrary variable $x$ that
	\[T \vdash_{\SE} c_m \cdot x^m:\bigwedge\Gamma_i \to syl(d_n, e_n \cdot s_1 \cdot ... \cdot s_n) \cdot c_m \cdot x^m:A.\]\\
	By \Cref{andlemma} we have $T \vdash_{\SE} \bigwedge x:\Gamma_i \to c_m \cdot x^m:\bigwedge \Gamma_i$ for all $i$.
	This leads to \[T \vdash_{\SE} \bigwedge x:\Gamma_i \to syl(d_n, e_n \cdot s_1 \cdot ... \cdot s_n) \cdot c_m \cdot x^m:A.\]\\
	Therefore $q = syl(d_n, e_n \cdot s_1 \cdot ... \cdot s_n) \cdot c_m \cdot x^m$.
\end{proof}

In the next definition, it is essential that both implications are justified by the same term. An axiomatically appropriate and schematic theory does not guarantee this.

\begin{definition}
A theory $T$ \emph{supports weakening} if  there exists a ground term $t$ such that for all formulas $A,B$
\[
t:(A \to A \lor B) \in T \quad\text{and}\quad t:(B \to A \lor B) \in T .
\]
\end{definition}
Note that supporting weakening is rather natural. For instance, it corresponds to accepting the $R\lor$ rule in Gentzen system $\mathsf{G3}$, see, e.g.~\cite{ts2000}. By this rule, we can infer from the (multi-)set $\{A,B\}$ both $A \lor B$ and $B \lor A$.
Thus from $A$ we get by weakening admissibility $\{A,B\}$ and thus both $A \lor B$ and $B \lor A$ by exactly the same reasoning. Therefore, the justification term representing this should also be the same.

\begin{definition}
	$A \in \Fml$ is a balanced disjunction of depth 0. If $A$ and $B$ are balanced disjuctions of depth $m$, then $A \vee B$ is a balanced disjunction of depth $m+1$.
\end{definition}

For Lemma~\ref{orlemma} and Lemma~\ref{rlemma} we assumed that a term~$d_n$ exists. Next we show that such terms do exist (for $n$ being a power of 2) if we work with a schematic theory and balanced disjunctions. This leads to the formulation of Lemma~\ref{r2lemma}, which is the same as  Lemma~\ref{rlemma} without the extra assumption about the term~$d_n$, but with a schematic theory that supports weakening. This provides the crucial step in the proof of the realization theorem.

\begin{lemma}\label{r2lemma}
	Let $T$ be an axiomatically appropriate and schematic theory that supports weakening and let $n \in \mathbb{N}_{>0}$.
	We assume that for all $i$ with $1 \leq i \leq n$,  we are given 
	multisets $\Gamma_i \subseteq \Fml$ with $|\Gamma_i|=m$ such that
     $T \vdash_{\SE} \bigwedge \Gamma_i \to A$.\\% for all $1 \leq i \leq n$.\\
	Then there exists $q \in \Tm$ such that  for each $\Gamma_i$  we have $T \vdash_{\SE} \bigwedge x: \Gamma_i \to q:A$, where $x$ is a variable.
\end{lemma}

\begin{proof}
We first show by induction that  for all $l \in \mathbb{N}$ terms $d_{2^l}$ exists such that for all formulas  $A_1,\ldots,A_{2^l}$
\begin{equation*}
T \vdash_{\SE} d_{2^l}:(A_i \to A_1 \vee \cdots \vee A_{2^l}),
\end{equation*}
     i.e.~the terms from~\eqref{eq:dn:1} exist for $n=2^l$.
	
	For the base case we note that term $d_1$ exists by the internalization property (this requires $T$ to be axiomatically appropriate).
	
	For the induction step, first observe that the term $d_2$ exists since $T$ supports weakening. 
	Now suppose 
	\[
	d_l:(A_i \to A_1 \vee ... \vee A_l) \text{ where $1 \leq i \leq l$}
	\]
	and 
	\[
	d_l:(A_i \to A_{l+1} \vee ... \vee A_{2l}) \text{ where $l+1 \leq i\leq 2l$}.
	\]
	Then $syl(d_l, d_2)$ serves as $d_{2l}$ because of
	\[
	d_2:(A_1 \vee ... \vee A_l \to (A_1 \vee ... \vee A_l) \vee  (A_{l+1} \vee ... \vee A_{2l}))
	\] 
	and 
	\[
	d_2:(A_{l+1} \vee ... \vee A_{2l} \to (A_1 \vee ... \vee A_l) \vee  (A_{l+1} \vee ... \vee A_{2l})). 
	\]
	The term  $syl(d_l, d_2)$ does not depend on the formulas $A_1,...,A_{2l}$ because $T$ is schematic.
	Therefore schematicness and the property of supporting weakening of $T$ imply the existence of $d_{2^l}$ for all $l \in \mathbb{N}$, where the disjunction is balanced of depth $l$.
	
	To finish the proof, we
	define $k \in \mathbb{N}$ such that $2^{k-1} < n \leq 2^k$ and then $\Gamma_i := \Gamma_1$ for $n+1 \leq i \leq 2^k$. Now we can apply \Cref{rlemma} and get $q \in \Tm$ such that $T \vdash_{\SE} \bigwedge x:\Gamma_i \to q:A$ for all $i \leq 2^k$, including all $i \leq n$. 
\end{proof}

In order to prove a realization theorem, we need a notion to relate different occurences of $\Box$ in a $\GK$-derivation. The main definition will be that of an essential family of $\Box$-occurences, which goes back to~\cite{Art01BSL} (for some examples see~\cite{justificationLogic2019}).

An instance of a $\GK$-rule relates to formulas $F$ and $G$ if either 
\begin{enumerate}
\item
the rule instance does not transform $F$ and $F=G$ or 
\item
$G$ results from $F$  in the application of the rule instance 
\end{enumerate}
For example in 
\[
\axiom{$(\supset \to)$}   \dfrac{A,\Gamma \supset \Delta, B}{\Gamma \supset \Delta, A \to B} 
\]
the formula $A$ in the premise is related to the formula $A \to B$ in the conclusion.

Let $\mathcal{D}$ be a derivation in $\GK$. We say that two occurrences of $\Box$ in $\mathcal{D}$ are \emph{related} if they occur at the same position in related formulas of premises and conclusions of a rule instance in $\mathcal{D}$;\footnote{that is, e.g., an occurrence of $\Box$ in the premise $A$ in the above instance of $(\supset \to)$ is related the the occurrence of $\Box$ at the same position in the subformula $A$ of $A\to B$ in the conclusion} we close this relationship of related occurrences under transitivity.

All occurrences of $\Box$ in $\mathcal{D}$ naturally split into disjoint \emph{families} of related occurrences. Note that the rules of $\GK$ preserve the polarity of related occurrences. Thus, all occurrences in a given family have the same polarity and we speak of \emph{positive} and \emph{negative families}, respectively. 

We call a family \emph{essential} if at least one of its members is introduced on the right-hand side of a $(\Box)$ rule.

\begin{theorem}[Realization]
	Let $T$ be a theory that is axiomatically appropriate, schematic and that supports weakening.\\ Then there exists a realization $r$ such that for all formulas $A$ of the language of modal logic, we have $\vdash_K A \Rightarrow T \vdash_{\SE} r(A)$.
\end{theorem}

\begin{proof}
	Let $\mathcal{D}$ be the $\GK$ derivation that proves $\supset A$. The realization $r$ is constructed in three steps:
	\begin{enumerate}
		\item Modify the derivation. For each essential family $f$ do the following: 
		If $n$ $(\Box)$ rules introduce a $\Box$-operator to $f$, their premises are $\Gamma_i \supset A$, $1 \leq i \leq n$, and none of the $\Gamma_i$ is empty then use first $(w\supset)$ to duplicate formulas of $\Gamma_i$ such that all $\Gamma_i$ have the same cardinality. After applying the $(\Box)$ rule remove the duplicates by $(c\supset)$.
		\item For each negative family and each non-essential positive family, replace all $\Box$ occurrences by the variable $x$.
		\item For each essential family $f$ do the following: Enumerate the $(\Box)$ rules from $1$ to $n$. The premises are $\Gamma_i \supset A$, where $1\leq i\leq n$. Step 1 guarantees either $|\Gamma_1| = ... = |\Gamma_n| =: m$ or $\Gamma_k = \emptyset$ for some $k$. In the first case construct a term $q$ according to \Cref{r2lemma}. In the second case a term $q$ can be found by the internalization property. Replace each $\Box$ of $f$ by $q$.
	\end{enumerate}
	We call the resulting derivation after these three steps $\mathcal{D}'$. Now we prove by induction on the $GK$-derivation that all sequents in $\mathcal{D}'$ are derivable in $\SE$ from the theory $T$.
	\begin{enumerate}
		\item $P\supset P$: $r(P \to P) = P \to P$ which is derivable is $\SE$.
		\item $\perp \supset$: $r(\perp \to \perp) = \perp \to \perp$ which is derivable is $\SE$.
		\item $(\Box)$: The case without empty $\Gamma_i$'s is covered in \Cref{r2lemma}. In the other case the premise is $\Gamma_j \supset A$. The I.H. for $\Gamma_k = \emptyset$ is $T \vdash_{\SE} r(A)$. Because the term for the introduced $\Box$ was constructed according to the internalization property this implies $T \vdash_{\SE} r(\Box A)$. By propositional reasoning we infer $T\vdash_{\SE} r(\bigwedge \Box \Gamma_i) \to r(\Box A)$ for all $i$, therefore $T \vdash_{\SE} r(\bigwedge \Box \Gamma_i \to \Box A)$ for all $i$.
		\item For the remaining rules the desired result is obtained by propositional reasoning. \qedhere
	\end{enumerate}
\end{proof}

However, the realization obtained by the previous theorem will not be normal.
In traditional justification logic normal realizations can be achieved using the 
sum-operation, which there (unlike in $\SE$) is axiomatized by 
\[
s:A \lor t:A \to s+t:A.
\]

Since we work with general theories (instead of simple constant specifications) and with variables that are interpreted universally, we can mimick the traditional sum-operation and perform the usual realization procedure given in~\cite{Art01BSL,justificationLogic2019}.
\begin{theorem}[Normal realization]
Let $T$ be an axiomatically appropriate and schematic theory such that for some constant~c
\[
x:A \to c\cdot x \cdot y: (A\lor B) \in T  \quad\text{and}\quad y:B \to c\cdot x \cdot y: (A\lor B)  \in T.
\]
Then there exists a normal realization~$r$ such that for all modal formulas $F$,
\[
\vdash_{\K} F \quad\text{implies}\quad T \vdash_{\SE} r(F).
\]
\end{theorem}

\begin{proof}
From the assumptions we get \[T \vdash_{\SE} s:A \to c \cdot s \cdot t:(A \vee A) \text{ and } T \vdash_{\SE} t:A \to c \cdot s \cdot t:(A \vee A).\] The internalization property delivers a term $u$, such that $T \vdash_{\SE} u:(A \vee A \to A)$. Applying \textbf{j} yields $T \vdash_{\SE} c \cdot s \cdot t:(A \vee A) \to u \cdot c \cdot s \cdot t:A$. Therefore we have \[T \vdash_{\SE} s:A \to u \cdot c \cdot s \cdot t:A \text{ and } T \vdash_{\SE} t:A \to u \cdot c \cdot s \cdot t:A\] and finally 
\[
T \vdash_{\SE} s:A \vee t:A \to u \cdot c \cdot s \cdot t:A.
\]
We can define the plus from traditional justification logic as $s * t := u \cdot c \cdot s \cdot t$. Because $T$ is schematic, there is one $u$ justifying all instances of $A \vee A \to A$, which ensures that $s*t$ doesn't depend on $A$.
Therefore Artemov's realization procedure~\cite{Art01BSL} can be used. We only need a small adjustment in the case of the ($\Box$) rule, which we show next.

Let an occurrence of a ($\Box$) rule have number $i$ in the enumeration of all $n_f$ ($\Box$) rules in a given family $f$. The corresponding node in the $\GK$ derivation $\mathcal{D}'$ is labelled by
\[\dfrac{B_1,...,B_n \supset A}{x_1:B_1,...,x_n:B_n \supset s_1*...*s_{n_f}:A}\]
where the $x$'s are variables, the $s$'s are terms and $s_i$ is a provisional variable. By the induction hypothesis we have
\[T \vdash_{\SE} B_1 \wedge ... \wedge B_n \to A.\]
It can be shown by induction on the derivation of $A$ that there exists a term $t$ such that
\[T \vdash_{\SE} x_1:B_1 \wedge ... \wedge x_n:B_n \to t:A.\]
Thus
\[T \vdash_{\SE} x_1:B_1 \wedge ... \wedge x_n:B_n \to s_1*...*s_{i-1}*t*s_{i+1}*...*s_{n_f}:A.\]
Substitute $t$ for $s_i$ everywhere in $\mathcal{D}'$. The built-in substitution property \textbf{jv} ensures that this doesn't affect the already established derivability results.
\end{proof}

\section{Applications}

The semiring interpretation of evidence has a wide range of applications. Many of them require a particular choice of the semiring.  The following are of particular interest to us (see also~\cite{Green:2017}):
\begin{itemize}
\item
$V = ([0,1],\max,\cdot,0,1)$ is called the Viterbi semiring. We can think of the elements of $V$ as \emph{confidence scores} and use them to model trust.
\item
$T = (\mathbb{R}^\infty_{+},\min,+,\infty,0)$ is called the tropical semiring. This is connected to \emph{shortest path problems}. In the context of epistemic logic, we can employ this semiring to model the costs of obtaining knowledge. Among other things, this might provide a fresh perspective on the logical omniscience problem, related to the approaches in~\cite{ArtKuz06CSL,ArtKuz09TARK,ArtKuz14APALnonote}.
\item
$P = (\mathcal{P}(S),\cup,\cap,\emptyset,S)$ is called the powerset lattice (semiring).
This is closely related to the recently introduced subset models for justification logic~\cite{StuderLehmannSubsetModel2019,StuderLehmannSubsetModelLFCS2020,exploringSM}.
This semiring can be used to model probabilistic evidence and aggregation thereof, see, e.g.,~\cite{Art16}.  
\item
$F = ([0,1],\max,\max(0,a+b-1),0,1)$ is called the  \L{}ukasiewicz semiring.  We can use it to model fuzzy evidences. Ghari~\cite{Ghari2016} provides a first study of fuzzy justification logic that is based on this kind of operations for combining evidence. 
\end{itemize}

Another stream of possible applications emerges from the fact that terms with variables represent actual functions.  If the underlying semiring is $\omega$-continuous, then the induced polynomial functions are $\omega$-continuous and, therefore, monotone~\cite{Kuich:1997}. Hence, they have least and greatest fixed points. Thus it looks very promising to consider this kind of semirings to realize modal fixed point logics like common knowledge. 

Common knowledge of a proposition $A$ is a fixed point of $\lambda X.(\mathsf{E}A \land \mathsf{E}X)$. There are justification logics with common knowledge available~\cite{Art06TCS,BucKuzStu11JANCL} but their exact relationship to modal common knowledge is still open.

We believe that non-wellfounded and cyclic proof systems~\cite{Afshari2017,baston/capretta:2018,BucKuzStu10M4M,Studer2008} are the right proof-theoretic approach to settle this question. In those systems, proofs can be regarded as fixed points and hence justifications realizing those cyclic proofs will be fixed points in a semiring. For this purpose, making use of formal power series, which might be thought of as infinite polynomials, look very promising.
First results in this direction have been obtained by Shamkanov~\cite{Shamkanov2016} who presents a realization procedure for G\"odel-L\"ob logic based on cyclic proofs.

\subsection*{Ackowledgements}
We thank the anonymous reviewers of the present journal version as well as the reviewers of the CLAR conference version of this paper for many helpful comments and suggestions for improvement.
We also thank Vladimir Krupski for noticing a bug in an earlier version of this paper.
We are grateful to Daniyar Shamkanov for the discussion about cyclic proofs and realization.
This research is supported by the Swiss National Science Foundation grant
200020\_184625.

\end{document}